\documentclass{amsart}
\usepackage{amsmath}%
\usepackage{amsthm}
\usepackage{amsfonts}%
\usepackage{amssymb}%
\usepackage{graphicx}
\usepackage{tikz-cd}
\usetikzlibrary{cd}
\usepackage{bm}
\usepackage{comment}
\usepackage{mathtools}
\usepackage[all,cmtip]{xy}

\usepackage[utf8]{inputenc}
\usepackage{a4}
\usepackage[hidelinks]{hyperref}
\usepackage{enumitem}

\usepackage{mathrsfs}

\DeclareFontFamily{OT1}{pzc}{}
\DeclareFontShape{OT1}{pzc}{m}{it}{<-> s * [1.10] pzcmi7t}{}
\DeclareMathAlphabet{\mathpzc}{OT1}{pzc}{m}{it}

\usepackage[onehalfspacing]{setspace}
\newtheorem{theorem}{Theorem}[section]

\newtheorem{lemma}[theorem]{Lemma}

\numberwithin{equation}{section}

\def\XXint#1#2#3{{\setbox0=\hbox{$#1{#2#3}{\int}$}
\vcenter{\hbox{$#2#3$}}\kern-.5\wd0}}

\newcommand{\R}{\mathbb{R}}

\newcommand{\N}{\mathbb{N}}
\newcommand{\Z}{\mathbb{Z}}

\newcommand{\Ha}{\mathcal{H}}
\newcommand{\leb}{\mathcal{L}}

\newcommand{\sgn}{\operatorname{sgn}}
\newcommand{\spt}{\operatorname{spt}}

\newcommand{\diam}{\operatorname{diam}}

\newcommand{\Lip}{\operatorname{Lip}}

\renewcommand{\epsilon}{\varepsilon}

\DeclareMathOperator{\bI}{\mathbf{I}}

\DeclareMathOperator{\bM}{\mathbf{M}}

\DeclareMathOperator{\md}{\operatorname{md}}
\DeclareMathOperator{\mass}{\mathbf{M}}

\newcommand{\bb}[1]{\llbracket #1\rrbracket}
\usepackage{stmaryrd}
\newcommand{\norm}[1]{\lVert#1\rVert}

\usepackage[colorinlistoftodos]{todonotes}
\usepackage{import}
\usepackage{xifthen}
\usepackage{pdfpages}
\usepackage{transparent}
\usepackage{xcolor}
\usepackage{stmaryrd}

\makeatletter
\newcommand*{\cone}{%
	{%
		\mathpalette\@coneOf{\times}%
	}%
}
\newcommand*{\@coneOf}[2]{%
	\sbox0{$\m@th#1\mathsf{#2}$}%
	\mathsf{#2}%
	\kern-\wd0 %
	\mkern2.00mu\relax
	\nonscript\mkern0.25mu\relax
	\mathsf{#2}%
}

\usepackage{etoolbox}

\makeatletter
\patchcmd{\@setaddresses}{\indent}{\noindent}{}{}
\patchcmd{\@setaddresses}{\indent}{\noindent}{}{}
\patchcmd{\@setaddresses}{\indent}{\noindent}{}{}
\patchcmd{\@setaddresses}{\indent}{\noindent}{}{}
\makeatother

\keywords{Lipschitz-volume rigidity, metric manifolds, integral currents}
\subjclass[2020]{Primary 53C23; Secondary 49Q15, 28A75}
\thanks{D.M. was supported by Swiss National Science Foundation grant 212867.}

\author{Denis Marti}

\address{Department of Mathematics\\ University of Fribourg\\  Chemin du Mus\'ee 23\\  1700 Fribourg, Switzerland}
\email{denis.marti@unifr.ch}

\title{The Lipschitz-volume rigidity problem for metric manifolds}

\begin{document}

\begin{abstract}
   We prove a Lipschitz-volume rigidity result for $1$-Lipschitz maps of non-zero degree between metric manifolds (metric spaces homeomorphic to a closed oriented manifold) and Riemannian manifolds. The proof is based on degree theory and recent developments of Lipschitz-volume rigidity for integral currents.
\end{abstract}

\maketitle

\section{Introduction}
In the study of Riemannian manifolds, Lipschitz-volume rigidity has long been considered a folklore result. In its most fundamental form, the result can be stated as follows.

\begin{theorem}\label{thme: lip-vol-rig-classic}
    Let $f\colon X \to M$ be a $1$-Lipschitz map of non-zero degree between two closed, oriented, connected Riemannian $n$-manifolds. If $\textup{Vol}(X)=\textup{Vol}(M)$, then $f$ is an isometric homeomorphism.
\end{theorem}

For a detailed proof and applications; see \cite{Burago-2,Burago-Sergei}. This result has led to new developments and research in various areas of mathematics. For example, in Riemannian geometry the assumption on the volume has been replaced by a condition on the scalar curvature, leading to the celebrated Lipschitz rigidity result for spin manifolds of Llarull \cite{Llarull}; see also \cite{Cechchini-Zeidler, Gromov-scalar-lectures}. More recently, there has been a growing interest in proving Lipschitz-volume rigidity type results in the field of analysis on metric spaces and geometric measure theory \cite{Basso-creutz-teri,basso2023geometric,Perales-DelNin,Damaris-Dimitrios,Zuest-23}. 

The proof of Theorem \ref{thme: lip-vol-rig-classic} relies on measure theoretic arguments and the differentiability of Lipschitz maps. Naively, this approach gives hope, that one could easily extend the result to manifolds with less differential structure. However, a simple counterexample shows that the result can fail even if $M$ is a geodesic Lipschitz manifold \cite{Zuest-23}. In the present work, we are interested in weakening the conditions on the domain manifold as much as possible. We prove the following result.


\begin{theorem}\label{thme: main}
    Let $n\geq 3$ and $M$ be a closed, oriented, connected Riemannian $n$-manifold. Let $X$ be a metric space homeomorphic to a closed, oriented, connected smooth $n$-manifold. Suppose that $\Ha^n(X) = \Ha^n(M)$ and porous sets in $X$ have measure zero. Then, every $1$-Lipschitz map of non-zero degree from $X$ to $M$ is an isometric homeomorphism.
\end{theorem}

We say that a Borel set $A \subset X$ is \textit{porous} if for each $x \in A$ there exist $\epsilon>0$ and a sequence $x_n \to x$ satisfying 
$$d(x_n,A) \geq \epsilon d(x_n,x)$$ 
for every $n \in \N$. By the Lebesgue density theorem, every porous set in a Riemannian manifold has measure zero. More generally, if the Hausdorff $n$-measure in $X$ is doubling or $X$ is a Lipschitz differentiability space, then porous sets in $X$ have measure zero \cite{Bate-Speight}. Moreover, if there exists $c>0$ such that $\Ha^n(B(x,r)) \geq cr^n$ for all $x \in X$ and every $r \in (0,\diam X)$, then every porous set in $X$ is negligible; see Lemma \ref{lemma: lower-bound-porous-sets-negl}.

Theorem \ref{thme: main} also holds under weaker assumptions on $X$. Indeed, it suffices to assume that $M$ is a Lipschitz manifold and an essential length space; see \cite{Zuest-23} for the definition. We do not know whether the condition on porous sets in $X$ is really necessary. Note that all manifolds considered in this work are assumed to be connected and without boundary. 

The author, in collaboration with Basso and Wenger, already established some Lipschitz-volume rigidity type results for metric manifolds \cite[Theorem 1.4]{basso2023geometric}. In fact, for surfaces, Theorem \ref{thme: main} holds without the assumption on porous sets. This has been proved in \cite{basso2023geometric}. Another proof using a different approach can be found in \cite{Damaris-Dimitrios}. However, both proofs of Theorem \ref{thme: main} for surfaces rely on a recent uniformization theorem for metric surfaces \cite{Ntalampekos-Romney-1, Ntalampekos-Romney-2}. In the case where $n \geq 3$, no such uniformization results are available and a different approach is required. In this context, it is known from \cite{basso2023geometric} that the  conclusion of Theorem \ref{thme: main} holds if $X$ is linearly locally contractible. Notice that the theorem presented here is more general. Indeed, by \cite[Theorem 4.1]{basso2023geometric}, linear local contractibility implies, that there exists $c>0$ such that $\Ha^n(B(x,r)) \geq cr^n$ for all $x \in X$ and every $r \in (0,\diam X)$. 
\medskip

We now explain the main steps in the proof of Theorem \ref{thme: main}. It is well-known that $f$ is mass-preserving in this context, that is, $\Ha^n(f(B)) = \Ha^n(B)$ for every Borel $B\subset X$. In the proof of Theorem \ref{thme: lip-vol-rig-classic}, one compares the volume of balls in $X$ and the volume of their images in $M$. If $X$ and $M$ are Riemannian, we have a uniform control on the volume of sufficiently small balls in both spaces simultaneously. Since $f$ is mass-preserving and $1$-Lipschitz, this implies that the image of two disjoint balls cannot overlap significantly. It follows that $f$ is a homeomorphism and locally bi-Lipschitz.
In the situation of Theorem \ref{thme: main}, we first use a result of Bate, see Theorem \ref{thme: weak-perturb-bate} below, to prove that $X$ is $n$-rectifiable. Therefore, in contrast to the proof of Theorem \ref{thme: lip-vol-rig-classic}, we obtain a uniform control of the measure of balls only after a Borel decomposition of $X$. This then implies that $f$ is bi-Lipschitz only on each part of this decomposition. However, this does not suffice to conclude that $f$ is an isometry. 
To overcome this issue, we use the theory of integral currents in the sense of Ambrosio-Kirchheim \cite{ambrosio-kirchheim-2000}. The condition on porous sets guarantees that degree theory can be applied. This allows us to use a refined version of the construction in \cite{basso2023geometric} and to obtain an integral cycle in $X$. Finally, the statement then follows from the recent work of Züst \cite{Zuest-23} on Lipschitz-volume rigidity for integral currents. Such an integral current should be interpreted as an analytic analog of the fundamental class of the space. It is sometimes called a metric fundamental class and is of independent interest. We refer to \cite{basso2023geometric} and \cite{Denis-Teri} for applications beyond Lipschitz-volume rigidity. 
 \medskip

 \noindent {\bf Acknowledgments:} I am very grateful to my supervisor Stefan Wenger for his valuable advice and support. Furthermore, I would like to thank the anonymous referee for carefully reading the paper and for their helpful comments.

\section{Preliminaries}
Let $(X,d)$ be a metric space. We write $B(x,r)=\{y \in X \colon d(x,y)< r\}$ for the open ball with center $x \in X$ and radius $r> 0$. For $U \subset X$ and $r>0$, we denote by 
$N_r(U) = \left\{ y \in X \colon d(y,U) < r\right\}$
the open $r$-neighborhood of $U$. A map $f\colon X \to Y$ between metric spaces is called \textit{$L$-Lipschitz} if 
$$d(f(x),f(y)) \leq L d(x,y)$$
for all $x,y \in X$. We write $\Lip(f)$ for the smallest constant satisfying this inequality. If $f$ is injective and also its inverse is $L$-Lipschitz we say $f$ is \textit{$L$-bi-Lipschitz}.

\subsection{Degree and orientation}
We record some basic facts about the degree we use later and refer to \cite{dold-1980} for a more detailed exposition of the topic. Let $f\colon X \to M$ be a continuous map between two oriented (not necessarily compact) topological $n$-manifolds $X$ and $M$. Recall that throughout this article all manifolds are assumed to be connected and without boundary. Given $x \in X$, the relative singular homology groups $H_n(U,U\setminus \{x\})$ and $H_n(V,V\setminus \{f(x)\})$ are infinite cyclic for any open neighborhood $U$ and $V$ of $x$ and $f(x)$, respectively. If $f(U) \subset V$ and $f(y) \neq f(x)$ for every $y \in U\setminus \{x\}$, then $f_* \colon H_n(U,U\setminus \{x\}) \to H_n(V,V\setminus \{f(x)\})$ is a well-defined homomorphism. It sends the generator of $H_n(U,U\setminus \{x\})$ to an integer multiple of the generator of $H_n(V,V\setminus \{f(x)\})$. This integer is called the \textit{local degree} of $f$ at $x$ and is denoted by $\deg(f,x)$. The local degree does not depend on the choice of neighborhoods $U$ and $V$. Now, let $g\colon X \to M$ be another continuous map with $f(x) = g(x)$ and let $U \subset X$ be an open neighborhood of $x$. If there exists a homotopy $H$ between $f$ and $g$ satisfying $H(y,t) \neq f(x)$ for all $y \in U \setminus \{x\}$ and every $t$, then $\deg(f,x) = \deg(g,x)$. Moreover, the local degree satisfies the following multiplicity property; see \cite[Chapter 8, Corollary 4.6]{dold-1980}. Suppose that $h \colon M \to N$ is a continuous map into an oriented topological $n$-manifold $N$. If the local degrees $\deg(f,x)$ and $\deg(h,f(x))$ are well-defined, then
$$\deg(h\circ f, x) = \deg(h,f(x)) \cdot \deg(f,x).$$
In case $X$ is a closed manifold, then the degree $\deg(f)$ of $f$ is well-defined. As the local degree, the degree is also a homotopy invariant. More precisely, if $g\colon X \to M$ is another continuous map that is homotopic to $f$, then $\deg(f)=\deg(g)$. It follows \cite[Chapter 8, Proposition 4.7]{dold-1980} that for every $y \in M$ for which the preimage $f^{-1}(y)$ is finite the local degree of $f$ at each $x \in f^{-1}(y)$ is well-defined and we have
$$\deg(f) = \sum_{x \in f^{-1}(y)} \deg(f,x).$$
We conclude with the following easy consequence of the homotopy invariance of the local degree. Let $M$ be an oriented smooth $n$-manifold and $f\colon \R^n \to M$ be continuous. If $f$ is differentiable at $x \in \R^n$ with non-degenerate differential $D_x f$, then $\deg(f,x)=\sgn(\det(D_x f))$.

\subsection{Rectifiability and porosity}

Let $X$ be a complete metric space. We denote by $\Ha^n$ the Hausdorff $n$-measure on $X$ and normalize it such that $\Ha^n$ equals the Lebesgue measure $\leb^n$ on $n$-dimensional Euclidean space. A $\Ha^n$-measurable set $E \subset X$ is said to be \textit{$n$-rectifiable} if there exist a countable number of Borel sets $K_i \subset \R^n$ and Lipschitz maps $\varphi_i \colon K_i \to X$ such that
\begin{equation}\label{eq: def-rect}
    \Ha^n\left( E \setminus \bigcup_i \varphi_i(K_i)\right) = 0.
\end{equation}

The following fundamental result about rectifiable sets in metric spaces goes back to Kirchheim; see \cite[Lemma 7.2]{bate-2020} for the version presented here.

\begin{lemma}\label{lemma: lip-bilip-decomp}
    Let $K \subset \R^n$ be a Borel set and $\varphi \colon K \to X$ be a Lipschitz map. Then, there exist countably many Borel sets $K_i \subset K$ such that
    \begin{enumerate}
        \item $\Ha^n(\varphi(K) \setminus \bigcup_i \varphi(K_i)) = 0$;
        \item $\varphi$ is bi-Lipschitz on each $K_i$.
    \end{enumerate}
\end{lemma}

It follows that for any $n$-rectifiable set $E \subset X$, there exists a countable collection of bi-Lipschitz maps $\varphi_i \colon K_i \subset \R^n \to X$ satisfying \eqref{eq: def-rect} and such that the $\varphi_i(K_i)$ are pairwise disjoint; compare to \cite[Lemma 4.1]{ambrosio-kirchheim-2000}. We call such a collection $\big\{\varphi_i,K_i\big\}$ an \textit{atlas} of $E$.
Furthermore, we say that a $\Ha^n$-measurable set $P \subset X$ is \textit{purely $n$-unrectifiable} if $\Ha^n(P \cap E)=0$ for every $n$-rectifiable set $E \subset X$.

\begin{lemma}\label{lemma: rect-unrect-decomp}
    Let $X$ be a metric space with finite Hausdorff $n$-measure. Then, there exists a decomposition $X = E \cup P$, where $E$ is $n$-rectifiable and $P$ is purely $n$-unrectifiable.
\end{lemma}

For the convenience of the reader, we include a proof of this standard fact. 

\begin{proof}
    Put 
    $$s = \sup\big\{ \Ha^n(A) \colon A \subset X \textup{ is } n \textup{-rectifiable}\big\}.$$
    For $i \in \N$, let $E_i \subset X$ be $n$-rectifiable such that $s-\frac{1}{i}< \Ha^n(E_i)$. We define $E = \bigcup_{i \in \N} E_i$. A countable union of $n$-rectifiable sets is itself $n$-rectifiable and hence, $E$ is $n$-rectifiable. Let $P = X \setminus E$ and let $B \subset X$ be $n$-rectifiable. Set $B' = P\cap B$. The definition of $s$ implies
    $$s- \frac{1}{i} <  \Ha^n(E_i) + \Ha^n(B') = \Ha^n(E_i \cup B' ) \leq s$$
    for all $i \in \N$. We conclude that $\Ha^n(P \cap B) = 0$ and in particular, $P$ is purely $n$-unrectifiable.
\end{proof}

We need a deep result due to Bate. The Euclidean norm is denoted by $\norm{\cdot}$.

\begin{theorem}\label{thme: weak-perturb-bate}
    Let $X$ be a complete metric space with finite Hausdorff $n$-measure and $f\colon X \to \R^N$ be Lipschitz. Let $P \subset X$ be purely $n$-unrectifiable. Then, for any $\epsilon>0$ there exists a $(\Lip(f)+\epsilon)$-Lipschitz map $g \colon X \to \R^N$ satisfying
    \begin{enumerate}
        \item$ \norm{f(x)-g(x)}\leq \epsilon$ for each $x \in X$ and $f(x)=g(x)$ whenever $d(x,P)\geq \epsilon $;
        \item $\Ha^n(g(P)) \leq \epsilon.$
    \end{enumerate}
\end{theorem}

\begin{proof}
    The theorem is a particular instance of \cite[Theorem 4.9]{bate-2020}. More precisely, it follows from \cite[Theorem 1.5]{bate24} that there exists a set $N \subset X$ with $\Ha^n(N) = 0$ such that $P\setminus N$ has an $n-1$ dimensional weak tangent field with respect to $f$. Therefore, \cite[Lemma 2.8]{bate-2020} implies that $P$ is contained in the set $\tilde{A}(f,n-1)$ appearing in the statement of \cite[Theorem 4.9]{bate-2020}. Moreover, by \cite[Observation 4.2]{bate-2020} the constant $\tilde{K}(\R^N,n-1)$ in \cite[Theorem 4.9]{bate-2020} is equal to $1$. This gives the statement of Theorem \ref{thme: weak-perturb-bate}.
\end{proof}

We conclude the section with the following lemma.

\begin{lemma}\label{lemma: lower-bound-porous-sets-negl}
    Let $X$ be a compact metric space with finite Hausdorff $n$-measure. Suppose that there exists $c>0$ such that $\Ha^n(B(x,r)) \geq c r^n$ for each $x\in X$ and every $r\in (0,\diam X)$. Then porous sets in $X$ have $\Ha^n$-measure zero.
\end{lemma}

\begin{proof}
    Let $A \subset X$ be a porous set. Fix $x \in A$ for the moment. Then there exists $\epsilon>0$ and a sequence $x_i  \to x$ satisfying
    $$d(x_i,A) \geq \epsilon d(x_i,x)$$
    for every $i \in \N$. For $i \in \N$, set $R_i = d(x_i,x)$ and $r_i = \frac{\epsilon}{2}d(x_i,x)$. It follows that $B(x_i,r_i) \subset B(x,2R_i)\setminus A$ for every $i \in \N$. Therefore,
    \begin{align*}
        \liminf_{i\to \infty}\frac{\Ha^n(B(x,2R_i)\setminus A)}{ 2^nR_i^n} \geq \liminf_{i\to \infty}\frac{\Ha^n(B(x_i,r_i))}{2^nR_i^n}\geq \liminf_{i\to \infty}\frac{c r_i^n}{2^nR_i^n} = \frac{\epsilon^n}{4^n} c>0.
    \end{align*}
    Recall that a porous set is a Borel set. Hence, \cite[2.10.19 (4)]{federer-gmt} implies that
    $$\lim_{r\to 0}\frac{\Ha^n(B(x,r)\setminus A)}{\omega_n r^n}= 0$$
    for almost all $x \in A$, where $\omega_n$ denotes the Lebesgue $n$-measure of the unit ball in $\R^n$. We conclude that $A$ has measure zero.
\end{proof}
\subsection{Metric currents}
Let $X$ be a complete metric space. We use the theory of metric currents developed by Ambrosio and Kirchheim \cite{ambrosio-kirchheim-2000}; see also \cite{lang-local}. For $k\geq 0$, we denote by $\mathcal{D}^k(X)$ the set of all $(k+1)$-tuples $(f,\pi_1,\dots,\pi_k)$ of Lipschitz functions \(f\colon X\to \R\) and \(\pi_i\colon X\to \R\), for \(i=1, \dots, k\), with \(f\) bounded. A metric $k$-current $T \in \bM_k(X)$ is a multilinear functional $T\colon \mathcal{D}^k(X) \to \R$ that satisfies three properties that are usually referred to as the continuity property, the locality property and the finite mass property. The last property guarantees the existence of a finite Borel measure $\norm{T}$ associated with $T$, called the mass measure of $T$. The \textit{support} of \(T\) is defined as
\[
\spt T=\big\{ x\in X : \norm{T}(B(x, r)) >0 \text{ for all } r>0\big\}.
\]
Let $\varphi\colon X \to Y$ be a Lipschitz map and $T \in \bM_k(X)$. The \textit{pushforward} of $T$ under $\varphi$ is the metric $k$-current in $Y$ defined as
$$\varphi_\# T(f,\pi_1,\dots,\pi_k) = T(f\circ \varphi,\pi_1\circ \varphi,\dots,\pi_k\circ\varphi)$$
for all $(f,\pi_1,\dots, \pi_{k})\in \mathcal{D}^k(Y)$. An important example of a $k$-current is given by the integration with respect to some $\theta \in L^1(\R^k)$. More specifically, each $\theta \in L^1(\R^k)$ induces a $k$-current $\bb{\theta} \in \bM_k(\R^k)$ defined as follows
$$\bb{\theta}(f,\pi_1,\dots,\pi_k) = \int_{\R^k} \theta f \det(D \pi) \; d\leb^k$$
for all $(f,\pi) = (f,\pi_1,\dots,\pi_k) \in \mathcal{D}^k(\R^k).$ Motivated by the definition of a rectifiable set in a metric space, a $k$-current $T \in \bM_k(X)$ is said to be \textit{integer rectifiable} if the following holds. There exist countably many compact sets $K_i \subset \R^k$, and $\theta_i \in L^1(\R^k,\Z)$ with $\spt \theta_i \subset K_i$ and bi-Lipschitz maps $\varphi_i \colon K_i \to X$ such that
\[
T=\sum_{i\in \N} \varphi_{i\#}\bb{\theta_i} \quad \text{ and } \quad \mass(T)=\sum_{i\in \N} \mass(\varphi_{i\#}\bb{\theta_i}).
\]
Any such collection $\{\varphi_i,K_i,\theta_i\}$ is called a \textit{parametrization} of $T$. 

Given a $k$-current $T$, the \textit{boundary} of $T$ is defined by
$$\partial T (f,\pi_1,\dots,\pi_{k-1}) = T(1,f,\pi_1,\dots,\pi_{k-1})$$ 
for all $(f,\pi_1,\dots, \pi_{k-1})\in \mathcal{D}^{k-1}(X)$. We say an integer rectifiable $k$-current is an integral current if $\partial T$ is a metric $(k-1)$-current as well. The space of all integral $k$-currents is denoted by $\bI_k(X)$. Finally, an integral current without boundary is called an integral cycle.

Let $M$ be a closed, oriented Riemannian $n$-manifold. There exists an integral $n$-cycle $\bb{M}$ naturally associated to $M$. This cycle is called the fundamental class of $M$ and is defined as follows \[
\bb{M}(f,\pi) = \int_Mf\det(D\pi)\,d\hspace{-0.14em}\Ha^n
\]
for all \((f, \pi)\in \mathcal{D}^n(M)\). It is a direct consequence of Stokes theorem, that this defines an integral current without boundary. We conclude this section with the recent Lipschitz-volume rigidity result of Züst for integral currents. We state a simpler version more suitable to our context.

\begin{theorem}\label{thme: lip-vol-rigidity-züst}{\normalfont(\cite[Theorem 1.2]{Zuest-23})}
    Let $M$ be a closed, oriented Riemannian $n$-manifold. Let $X$ be a complete metric space and $T \in \bI_n(X)$ an integral cycle with $\mass^{\rm b}(T) \leq \Ha^n(M)$. Then, every $1$-Lipschitz map $f\colon \spt T \to M$ satisfying $f_\# T = \bb{M}$ is an isometry.
\end{theorem}

Given an integer rectifiable current $T$ and a parametrization $\{\varphi_i,K_i,\theta_i\}$, we define the Hausdorff or Busemann mass $\mass^{\rm b}(T)$ of $T$ by

$$
\mass^{\rm b}(T) =\sum_{i\in\N}\int_{K_i}|\theta_i| \,\mathbf{J}(\md\varphi_i)\,d\mathscr{L}^n.
$$

Here, $\md\varphi_i$ denotes the metric differential of $\varphi_i$. It follows from the area formula for Lipschitz maps \cite{kirchheim1994rectifiable} that an integer rectifiable current $T$ and its Hausdorff or Busemann mass do not depend on the chosen parametrization. We refer to \cite{zust2021riemannian} and \cite{Zuest-23} for more information.

\section{Proof of the main theorem}
Throughout this section, let $M$ be a closed, oriented Riemannian $n$-manifold and $X$ a metric space homeomorphic to a closed, oriented smooth $n$-manifold. Suppose that $\Ha^n(X)=\Ha^n(M)$ and that there exists a $1$-Lipschitz map $f\colon X \to M$ of non-zero degree.

\medskip We first establish some properties of $X$ and $f$ that hold without the additional assumption that porous sets in $X$ have measure zero. The first lemma states that $f$ is mass-preserving and almost injective in the measure-theoretic sense. This is a standard fact. Nevertheless, we give the simple proof because we use it frequently.

\begin{lemma}\label{lemma: mass-preserving-almost-injective}
    For every Borel set $B \subset X$ we have $\Ha^n(f(B)) = \Ha^n(B)$. Moreover, for almost all $y \in M$ the preimage $f^{-1}(y)$ consists of a single point.
\end{lemma}

\begin{proof}
    The map $f$ has non-zero degree and is therefore surjective. This, together with the fact that $f$ is $1$-Lipschitz, implies that $f$ is mass-preserving. Indeed, for $B \subset X$ Borel we have

    $$\Ha^n(M) = \Ha^n(f(X))  \leq\Ha^n(f(X\setminus B)) +\Ha^n(f(B))\leq \Ha^n(X\setminus B)+\Ha^n(B) = \Ha^n(X).$$
    Since $\Ha^n(X)=\Ha^n(M)$, each inequality becomes an equality. Therefore, for every Borel set $B\subset X$ we have $\Ha^n(f(B)) = \Ha^n(B)$. Furthermore, it follows from the coarea inequality \cite[2.10.25]{federer-gmt} that
    $$\Ha^n(B) = \Ha^n(f(B)) \leq \int_M \Ha^0(f^{-1}(y) \cap B) \; d\Ha^n(y) \leq \Ha^n(B)$$
    for every Borel $B \subset X$. We conclude that for almost all $y\in M$ the preimage $f^{-1}(y)$ consists of a single point. 
\end{proof}


Next, we show that $X$ is $n$-rectifiable by proving that every purely $n$-unrectifiable subset $P \subset X$ has measure zero. Since $f$ is mass-preserving, it is sufficient to show that $f(P)$ has measure zero. We outline the ideas to achieve this. First, we embed $M$ into Euclidean space via a bi-Lipschitz map $\iota$. Then, using Bate’s result (Theorem \ref{thme: weak-perturb-bate}) together with a Lipschitz retraction of $M$, we perturb $g =\iota \circ  f$ in a small neighborhood of $P$. This gives a surjective Lipschitz map $h$ from $X$ to $M$ such that $h(P)$ has small measure. Because $g$ is almost injective and $h$ is close to $g$, the measures of $g(P)$ and $h(P)$ are comparable. Since $h(P)$ can be made arbitrarily small and $h$ can be made arbitrarily close to $g$, it follows that $g(P)$ and in particular, $f(P)$ have measure zero. 


\begin{lemma}\label{lemma: X-rect-new}
    The space $X$ is $n$-rectifiable.
\end{lemma}

\begin{proof}
    The Riemannian manifold $M$ is closed. Therefore, there exist a bi-Lipschitz embedding $\iota\colon M \xhookrightarrow{} \R^N$ and a Lipschitz retraction $\pi \colon N_\delta(\iota(M)) \to \iota(M)$ for some $N\geq n$ and $\delta>0$. See e.g. \cite[Theorem 3.10]{Heinonen-geom-embeddings} and \cite[Theorem 3.1]{hohti-1993}. Set $g= \iota \circ f$. 
    Lemma \ref{lemma: rect-unrect-decomp} implies that there exists a decomposition $X = E \cup P$, where $E$ is $n$-rectifiable and $P$ is purely $n$-unrectifiable. We have to show that $\Ha^n(P)=0$. By Lemma \ref{lemma: mass-preserving-almost-injective} the preimage $g^{-1}(y) = (\iota \circ f)^{-1}(y)$ consists of a single point for almost all $y \in \iota(M)$. Removing a set of measure zero from $P$ if necessary, we may suppose that $g^{-1}(g(P)) = P$. It follows from \cite[Theorem 2.2.2(1) and 2.2.3]{federer-gmt} that for every $\epsilon>0$ there exists $P' \subset P$ closed with $\Ha^n(P \setminus P') < \epsilon$. In particular, it suffices to show that each closed subset $P' \subset P$ has $\Ha^n$-measure zero in order to show that $P$ has $\Ha^n$-measure zero. Hence, we may further suppose that $P$ is closed. Let $0<\epsilon<\delta$. We apply Theorem \ref{thme: weak-perturb-bate} to $g$ and obtain a $2\Lip(g)$-Lipschitz map $\tilde{g}\colon X \to \R^N$ satisfying
    \begin{enumerate}
        \item $\norm{g(x)-\tilde{g}(x)}\leq \epsilon'$ for each $x \in X$ and $\tilde{g}(x)=g(x)$ whenever $d(x,P)\geq \epsilon'$;
        \item $\Ha^n(\tilde{g}(P)) \leq \epsilon'$.
    \end{enumerate}
    Here, $\epsilon'>0$ is such that $\epsilon' \Lip(\pi)^n < \epsilon$. Since $\pi$ is a retraction we have $\Lip(\pi) \geq 1$ and thus $\epsilon' <\epsilon$. It follows that $h =  \pi \circ \tilde{g}$ is well-defined, $(2 \Lip(g) \Lip(\pi))$-Lipschitz and satisfies
    \begin{enumerate}[label=(\roman*)]
        \item $\norm{g(x)-h(x)}\leq \epsilon$ for each $x \in X$ and $h(x)=g(x)$ whenever $d(x,P)\geq \epsilon$;
        \item $\Ha^n(h(P)) \leq \epsilon$.
    \end{enumerate}
    Notice that (1) implies (i) because $(\pi \circ g)(x) = g(x)$ for every $x \in X$. Moreover, the straight-line homotopy $H$ between $g$ and $h$ satisfies $H(x,t) \in N_\delta(\iota(M))$ for every $x \in X$ and every $t\in [0,1]$. Therefore, the composition of $H$ and $\pi$ is well-defined and $g$ and $h$ are homotopic as maps from $X$ to $\iota(M)$. We conclude that $h$ is surjective.    
    Recall that $g^{-1}(g(P)) = P$. Hence, the second part of (i) implies that $h^{-1}(g(P)) \setminus N_\epsilon(P)= \emptyset$ and in particular, $h^{-1}(g(P)) \subset N_\epsilon(P)$. It follows that $h(h^{-1}(g(P))) = g(P) \subset h(N_\epsilon(P))$ because $h$ is surjective. Thus,
    

    $$\Ha^n(g(P)) \leq \Ha^n(h(N_\epsilon(P))) \leq \epsilon + \Lip(h)^n \cdot\Ha^n(N_\epsilon(P) \setminus P).$$
    Furthermore, using that $f$ is mass-preserving and $\iota$ is bi-Lipschitz, we get
    $$\Ha^n(P) = \Ha^n(f(P)) \leq \Lip(\iota^{-1})^n \cdot\Ha^n(g(P)).$$
    Since $P$ is closed we have $\Ha^n(N_\epsilon(P) \setminus P)\to 0$ as $\epsilon\to 0$. We conclude that $\Ha^n(P)=0$. This completes the proof. 
\end{proof}

We need the following fundamental property of rectifiable sets in metric spaces. By Kirchheim \cite[Theorem 9]{kirchheim1994rectifiable} the \textit{$n$-dimensional density }
$$\lim_{r \to 0}\frac{\Ha^n(B(x,r))}{\omega_n r^n} $$
exists and is equal to $1$ for almost all $x \in X$. Here, $\omega_n$ denotes the volume of the Euclidean unit $n$-ball. It is easy to prove that replacing the open ball $B(x,r)$ with a closed ball $\overline{B}(x,r)$ in the definition of the density does not affect the value of the limit.

\begin{lemma}\label{lemma: f-bi-lip-decomp-new}
    There exist countably many compact sets $C_i \subset X$ with the following properties:
    \begin{enumerate}
        \item $\Ha^n(X \setminus \bigcup_i C_i) = 0$;
        \item for every $i \in \N$ and all $x,y\in C_i$ we have
            $$ \frac{1}{2}  d(x,y) \leq d(f(x),f(y)) \leq d(x,y).$$
    \end{enumerate}
    Moreover, we may assume that for each $i \in \N$, there exists $K_i \subset \R^n$ compact and  $\varphi_i\colon K_i \to X$ bi-Lipschitz with $\varphi_i(K_i)=C_i$. 
\end{lemma}  
\begin{proof}
    For $r>0$ and $x \in X$ define
    $$g_r(x) = \sup_{0<s< r}\frac{\Ha^n(B(x,s))}{\omega_n s^n} \quad \textup{and} \quad h_r(x) = \inf_{0<s<r}\frac{\Ha^n(\overline{B}(x,s))}{\omega_n s^n}.$$
    Since $X$ is $n$-rectifiable we conclude that $g_r(x), h_r(x) \to 1$ as $r \to 0$ for almost all $x\in X$. Moreover, the functions $g_r$ and $h_r$ are lower semicontinuous and upper semicontinuous, respectively; see \cite[Chapter 1, Remark 3.2]{simon2014introduction}. It follows from Egoroff’s theorem \cite[2.3.7]{federer-gmt} and the Borel regularity of $\Ha^n$ that for every $\epsilon>0$, there exists a compact set $C\subset X$ such that $\Ha^n(X \setminus C) < \epsilon$ and both $g_r$ and $h_r$ converge uniformly on $C$. By applying this observation repeatedly to a sequence of $\epsilon>0$ converging to zero we conclude the following. There exist countably many compact sets $C_i \subset X$ such that
    \begin{enumerate}[label=(\roman*)]
        \item $\Ha^n(X \setminus \bigcup_i C_i ) =0$;
        \item for every $i \in \N$ and every $\delta>0$ there exist $r>0$ such that 
        $$1-\delta \leq \inf_{0<s < r}\frac{\Ha^n(\overline{B}(x,s))}{\omega_n s^n} \leq \sup_{0<s < r}\frac{\Ha^n(B(x,s))}{\omega_n s^n} \leq 1+\delta $$
        for each $x \in C_i$.
    \end{enumerate}
    Fix $i \in \N$ for the moment. Let $\delta>0$ be such that $4^n 5 \delta<1$. By the coarea inequality we have 
    $$\int_0^r \Ha^{n-1}(\{y\in X \colon d(x,y)=s\}) \;ds \leq \Ha^n(B(x,r))<\infty,$$
    for every $x \in X$ and each $r>0$. It follows that the $\Ha^n$-measure of the open $B(x,s)$ and the closed ball $\overline{B}(x,s)$ are equal for every $x \in X$ and almost every $s>0$. Therefore, (ii) implies that there exists $r>0$ such that
    \begin{equation}\label{eq: bound-volume-balls}
        \omega_n s^n - \delta \omega_n s^n  \leq \Ha^n(B(x,s)) \leq \omega_n s^n + \delta \omega_n s^n 
    \end{equation}
    for every $x \in C_i$ and almost all $s \in (0,r)$. Since $M$ is a closed Riemannian manifold, we may assume that \eqref{eq: bound-volume-balls} also holds for every $x \in M$ and each $s \in (0,r)$. Now, let $x,y \in C_i$ with $d(x,y) <2r$ and suppose that $d(f(x),f(y)) < \frac{1}{2}d(x,y)$. Set $t =  d(f(x),f(y))$ and let $s \in (\frac{3}{4}t,t)$ be such that \eqref{eq: bound-volume-balls} holds for $B(x,s)$ and $B(y,s)$. By the above and using that $f$ is mass-preserving, we get
    \begin{equation*}\label{eq: f-local-bilip-1}
        \Ha^n(f(B(x,s)\cup B(y,s))) = \Ha^n(B(x,s) \cup B(y,s)) \geq 2\omega_n s^n - 2\delta \omega_n s^n.
    \end{equation*}
    On the other hand, we have $f(B(x,s)\cup B(y,s)) \subset B(f(x),s) \cup B(f(y),s)$, because $f$ is $1$-Lipschitz. Furthermore, the intersection $B(f(x),s)\cap B(f(y),s)$ contains a ball of radius $s/4$. Therefore,
    $$  \Ha^n(f(B(x,s)\cup B(y,s))) \leq  \left(2-\frac{1}{4^n}\right)\omega_ns^n + \left(2+\frac{1}{4^n}\right)\delta \omega_n s^n.$$
    However, by our choice of $\delta$ we have
    $$2\omega_n s^n - 2\delta \omega_n s^n >  \left(2-\frac{1}{4^n}\right)\omega_ns^n +\left(2+\frac{1}{4^n}\right) \delta \omega_n s^n.$$
    This is clearly a contradiction. We conclude that $\frac{1}{2}d(x,y) \leq d(f(x),f(y)) $ for all $x,y \in C_i$ with $d(x,y) \leq 2r$. The claim follows by further partitioning each $C_i$ using the $n$-rectifiability of $X$.
\end{proof}

The next lemma is the only place where we use the fact that porous sets in $X$ have measure zero. Specifically, we use the following consequence of this fact. Let $A \subset X$ be Borel and $\epsilon>0$, then for almost all $x \in A$ there exists $r>0$ such that for every $y \in B(x,r)$ there exists $z \in A$ satisfying $d(y,z) \leq \epsilon d(y,x)$.

\begin{lemma}\label{lemma: porous-degree-lipschitz}
    Let $g,h \colon X \to \R^n$ be Lipschitz and $C \subset X$ be a Borel set such that $g(x)=h(x)$ for all $x \in C$. If porous sets in $X$ have measure zero, then for almost all $y \in \R^n$ and every $x \in g^{-1}(y) \cap C$ we have $\deg(g,x) = \deg(h,x)$. 
\end{lemma}

\begin{proof}
   First, we prove the statement in the special case when $C = \varphi(K)$ for some compact set $K \subset \R^n$ and a bi-Lipschitz map $\varphi\colon K \to X$. By applying Lemma \ref{lemma: lip-bilip-decomp} to 
    $\eta = g \circ \varphi = h \circ \varphi$ and using that $\varphi$ is bi-Lipschitz we obtain countably many compact sets $C_i \subset C$ such that
    \begin{itemize}
        \item $\Ha^n(g(C) \setminus \bigcup_i g(C_i)) = \Ha^n(h(C) \setminus \bigcup_i h(C_i)) = 0$;
        \item $g$ and $h$ are bi-Lipschitz on each $C_i$.
    \end{itemize}
    For $i \in \N$, let $N_i\subset C_i$ be the set of points $x\in C_i$ such that $C_i$ is porous at $x$, that is, the property stated before the lemma does not hold at $x$. By our assumption $N = \bigcup_i N_i$ has measure zero and in particular, $\Ha^n(g(N)) = \Ha^n(h(N)) = 0$. Now, let $y \in \R^n$ be such that $g^{-1}(y) \subset \bigcup_i C_i\setminus N$ and the preimage $g^{-1}(y)$ is finite. Notice that by the above and the coarea inequality, this holds for almost all $y \in \R^n$. Let $x \in g^{-1}(y)$ and $i \in \N$ with $x \in C_i$. Then there exists $L\geq 1$ such that $g$ and $h$ are $L$-Lipschitz on $X$ as well as $L$-bi-Lipschitz on $C_i$. Furthermore, there exists $r>0$ with the following property. For every $z \in B(x,r)$ there exists $z' \in C_i$ satisfying $d(z,z')\leq (3L)^{-2} d(x,z)$. For such $z$ we have
    \begin{equation}\label{eq: degree-ineq-1}
        \norm{g(x)-g(z)} \geq \norm{g(x)-g(z')} - \norm{g(z)-g(z')} \geq\frac{7}{9L}d(x,z).
    \end{equation}
    Notice that the same holds if $g$ is replaced by $h$. Moreover, 
    \begin{equation}\label{eq: degree-ineq-2}
        \norm{g(z)-h(z)} \leq \norm{g(z)-g(z')} + \norm{h(z')-h(z)} \leq 2L d(z,z') \leq \frac{2}{9L}d(x,z).
    \end{equation}
    
    It follows from \eqref{eq: degree-ineq-1} that for all $z \in B(x,r) \setminus \{x\}$ we have $g(z) \neq g(x)$ and $h(z) \neq h(x)$. Therefore, the local degrees $\deg(g,x)$ and $\deg(h,x)$ are well-defined. Furthermore, let $H\colon X\times [0,1] \to \R^n, (z,t) \mapsto t g(z) + (1-t) h(z)$ be the straight-line homotopy between $g$ and $h$. By \eqref{eq: degree-ineq-2} we have 
    $$\norm{H(z,t)-g(z)} \leq \norm{(1-t) (g(z)-h(z))} \leq (1-t) \frac{2}{9L}d(x,z)$$
    for every $z \in B(x,r)$ and every $t\in [0,1]$. Together with \eqref{eq: degree-ineq-1} we get
    $$\norm{H(z,t)-g(x)} \geq \norm{g(x)-g(z)} - \norm{H(z,t)-g(z)} \geq (1-t) \frac{5}{9L}d(x,z).$$
    This shows that $H$ satisfies $ H(z,t)\neq g(x)$ for every $z \in B(x,r)\setminus \{x\}$ and every $t\in [0,1]$. We conclude from the homotopy invariance of the local degree that $\deg(g,x)= \deg(h,x)$. This completes the proof in the special case. We now prove the general statement. Since $X$ is $n$-rectifiable, $C\subset X$ is $n$-rectifiable as well. Thus, there exist countably many compact sets $K_i \subset \R^n$ and bi-Lipschitz maps $\varphi_i \colon K_i \to X$ such that the $C_i = \varphi_i(K_i)$ are pairwise disjoint and $N = C \setminus \bigcup_{i \in \N} C_i$ has $\Ha^n$-measure zero. We have 
    $$g^{-1}(y) \cap C = g^{-1}(y)\cap \left(\bigcup_{i \in \N} C_i\right) = \bigcup_{i\in\N} g^{-1}(y)\cap C_i$$ 
    for almost all $y \in \R^n$ because $g(N) = h(N)$ has $\Ha^n$-measure zero. Therefore, the first part of the proof implies that for almost all $y \in \R^n$ and every $x \in g^{-1}(y) \cap C$ the local degrees $\deg(g,x)$ and $\deg(h,x)$ coincide. This completes the proof. 
\end{proof}

We are now in a position to construct a non-trivial integral cycle in $X$. 

\begin{lemma}\label{lemma: existence-cycle-new}
    If porous sets in $X$ have measure zero, then there exists an integral cycle $T \in \bI_n(X)$. Moreover, the parametrization $\{\varphi_i,K_i,\theta_i\}$ of $T$ is such that $\{\varphi_i,K_i\}$ is an atlas of $X$ and $|\theta_i(x)|=1$ for every $i\in\N$ and almost all $x \in K_i$.
\end{lemma}

\begin{proof}
    Recall that $X$ is $n$-rectifiable. Lemma \ref{lemma: f-bi-lip-decomp-new} implies there exists an atlas $\big\{ \varphi_i,K_i\big\}$ of $X$ 
    such that $f$ is $2$-bi-Lipschitz on each $C_i = \varphi_i(K_i)$. For $i \in \N$, set $g_i = f \circ \varphi_i$. By partitioning each $K_i$ if necessary, we can suppose that each $f(\varphi_i(K_i))$ is contained in an open ball in $M$ that is bi-Lipschitz homeomorphic to an open ball in $\R^n$. Using McShane's extension theorem, we conclude that, for each $i\in \N$, there exists a Lipschitz extension $\tilde{g}_i$ of $g_i$ to some open neighborhood $U_i$ of $K_i$. Since every $\tilde{g_i}$ is Lipschitz and bi-Lipschitz on $K_i$, it follows that for each $i \in \N$ and for almost all $z \in K_i$, the extension $\tilde{g}_i$ is differentiable and the differential $D_z \tilde{g}_i$ is non-degenerate. Therefore, for such $z$ we have $\deg(\tilde{g_i},z) = \sgn(\det(D_z \tilde{g}_i))$ and in particular, the local degree $\deg(\tilde{g_i},z)$ equals $1$ or $-1$. Define
    $$T= \sum_i \varphi_{_i\#}\bb{\theta_i},$$
    where $\theta_i(z)= \deg(\tilde{g_i},z)$. One shows exactly as in the proof of \cite[Lemma 5.1]{basso2023geometric} that $T$ is an integer rectifiable $n$-current of finite mass. It remains to prove that $\partial T =0$. More precisely, we have to show that
    $$T(1,\pi) = \sum_i \varphi_{i\#} \bb{\theta_i}(1,\pi) =0$$
    for all $\pi \in \Lip(X,\R^n)$. Fix $i \in \N$ and $\pi \in \Lip(X,\R^n)$ for the moment. Let $\psi \colon \R^n \to \R^n$ be a Lipschitz extension of $\pi \circ \varphi_i$ and $\eta\colon M \to \R^n$ be a Lipschitz extension of $\pi \circ f^{-1} \colon f(C_i) \to \R^n$. Notice that $\eta$ satisfies $\eta \circ \tilde{g}_i (z) = \pi \circ \varphi_i(z)$ for all $z \in K_i$. In particular, $\eta \circ \tilde{g_i}$ is another Lipschitz extension of $\pi \circ \varphi_i $.

    \[
    \begin{tikzcd}[row sep=huge, column sep=huge,
  nodes={scale=1.3, text depth=0pt},nodes={inner xsep=0.32em}]
    U_i
        \arrow[drrr,bend left=37.5,"\tilde{g}_i"]
        \arrow[d,hookleftarrow]
    &
    &
    &
    \\
    K_i 
        \arrow[r,"\varphi_i"]
        \arrow[d,hookrightarrow]
        \arrow[rrr,bend left=35, "g_i"]
    & X
        \arrow[d,"\pi"]
        \arrow[rr,bend left=20,"f"]
    & f(C_i)
        \arrow[r,hookrightarrow]
        \arrow[l,"f^{-1}"]
    & M
        \arrow[dll,bend left=25,"\eta"]
    \\
    \R^n
        \arrow[r,"\psi"]
    & \R^n
    &
    &
    \end{tikzcd}
    \]


    Let $A \subset M$ be the set of points $x \in M$ where $\eta$ is not differentiable. By Rademacher's theorem $A$ has $\Ha^n$-measure zero. Using that $f$ is mass-preserving and surjective we get
    $$\Ha^n(f^{-1}(A)) = \Ha^n(f(f^{-1}(A))) = \Ha^n(A) = 0.$$
    We conclude that $\varphi_i^{-1}(f^{-1}(A))$ has $\Ha^n$-measure zero because $\varphi_i$ is bi-Lipschitz. Furthermore, it is not difficult to show that if $z\in K_i$ is a density point of $K_i$ and the derivatives $D_z(\eta \circ \tilde{g_i})$ and $D_z\psi$ exist, then they are equal. We define $B_1 \subset K_i$ to be the set of points $z$ such that the following holds. The Lipschitz maps $\psi$, $ \tilde{g_i}$ and $\eta$ are differentiable at $z$ and $\tilde{g}_i(z)$, respectively, as well as $D_z(\eta \circ \tilde{g_i}) = D_z\psi$. Since almost every point in $K_i$ is a density point of $K_i$ and $\Ha^n(\varphi_i^{-1}(f^{-1}(A))) = 0$, we have $\Ha^n( K_i \setminus B_1) = 0$. Moreover, let $B_2 \subset B_1$ be the set of points $z$ such that the derivative $D_z\psi$ is non-degenerate. The area formula implies
    $$\Ha^n(\psi(K_i\setminus B_2)) \leq \int_{\R^n} \Ha^0((K_i\setminus B_2) \cap \psi^{-1}(y)) \;\leb^n(y)= \int_{K_i\setminus B_2} \det(D_z\psi) \;\leb^n(z) =0.$$
    Finally, it follows from the coarea inequality that the preimages $\psi^{-1}(y)\cap K_i$ and $(\eta  \circ \tilde{g}_i)^{-1}(y)\cap K_i$ are finite for almost all $y\in \R^n$. The multiplicity of the degree, when combined with the preceding observations, implies
    \begin{equation*}\label{eq: degree-equal-extensions}
         \sgn(\det  D_z\psi)= \sgn(\det  D_z(\eta \circ \tilde{g_i})) = \deg(\eta \circ \tilde{g_i},z) =  \deg(\tilde{g_i},z) \cdot \deg(\eta,\tilde{g_i}(z))
     \end{equation*}
    for almost all $y \in \R^n$ and every $z \in \psi^{-1}(y)\cap K_i$. Therefore,

    \begin{align*}
        \varphi_{i\#}\bb{\theta_i}(1,\pi) &= \int_{K_i} \theta_i(z) \det(D_z \psi) \; d\leb^n(z)
        \\
        &= \int_{K_i} \theta_i(z) \sgn(\det(D_z \psi)) |\det(D_z \psi)| \; d\leb^n(z)
        \\
        &= \int_{K_i} \deg(\eta,\tilde{g_i}(z))  |\det(D_z \psi)| \; d\leb^n(z).
    \end{align*}

    The area formula yields


    $$  \int_{K_i} \deg(\eta,\tilde{g_i}(z))  |\det(D_z \psi)| \; d\leb^n(z)= \int_{\R^n}\left(\sum_{z\in \psi^{-1}(y) \cap K_i} \deg(\eta, \tilde{g}_i(z))\right)d\leb^n(y).$$

    Recall that $C_i = \varphi_i(K_i)$ and $\psi(z) = (\pi \circ \varphi_i)(z)$ for every $z \in K_i$. Using that $\varphi_i$ is bi-Lipschitz on $K_i$ we get $z \in \psi^{-1}(y)\cap K_i$ if and only if $\varphi_i(z) \in \pi^{-1}(y)\cap C_i$ for every $y \in \R^n$. Hence,

    \begin{equation}\label{eq: existence-integral-current-1}
        \sum_{z\in \psi^{-1}(y) \cap K_i} \deg(\eta, \tilde{g}_i(z)) = \sum_{\varphi_i(z) \in \pi^{-1}(y)\cap C_i} \deg(\eta, \tilde{g}_i(\varphi_i(z)))
    \end{equation}
    
    for all $y \in \R^n$. For $z \in K_i$, we have $\tilde{g}_i(z) =(f\circ\varphi_i)(z)$ and thus, $\deg(\eta,\tilde{g}_i(z)) = \deg(\eta,f(x))$ for $x= \varphi_i(z)\in C_i$. Therefore,
    
    \begin{equation}\label{eq: existence-integral-current-2}
            \sum_{\varphi_i(z) \in \pi^{-1}(y)\cap C_i} \deg(\eta, \tilde{g}_i(\varphi_i(z))) = \sum_{x\in \pi^{-1}(y) \cap C_i} \deg(\eta, f(x))
    \end{equation}

    for every $y \in\R^n$. Combining \eqref{eq: existence-integral-current-1} and \eqref{eq: existence-integral-current-2} we obtain


    $$       \int_{\R^n}\left(\sum_{z\in \psi^{-1}(y) \cap K_i} \deg(\eta, \tilde{g}_i(z))\right)d\leb^n(y)  = \int_{\R^n}\left(\sum_{x\in \pi^{-1}(y) \cap C_i} \deg(\eta, f(x))\right)d\leb^n(y).$$

    By Lemma \ref{lemma: mass-preserving-almost-injective} the preimage $f^{-1}(z)$ consists of a single point for almost all $z \in M$ and $f$ is mass-preserving. Hence, the additivity property of the degree implies that $\deg(f,x) = \deg(f)$ for almost all $x \in X$. Moreover, for each $x \in C_i$ we have $(\eta\circ f)(x) = \pi(x)$. It follows from Lemma \ref{lemma: porous-degree-lipschitz} and the multiplicity of the degree that 
    \begin{equation*}\label{eq: degree-pi-f-eta}
        \deg(\pi,x) = \deg(\eta\circ f,x) =  \deg(\eta,f(x))\cdot \deg(f)
    \end{equation*}
    for almost all $y\in \R^n$ and every $x \in \pi^{-1}(y) \cap C_i $. In particular,
    

    \begin{align*}
         &\int_{\R^n}\left(\sum_{x\in \pi^{-1}(y) \cap C_i} \deg(\eta, f(x))\right)d\leb^n(y)=
        \\
         &\frac{1}{\deg(f)}\int_{\R^n}\left(\sum_{x\in \pi^{-1}(y) \cap C_i} \deg(\pi, x)\right)\,d\leb^n(y).
    \end{align*}
        

    We conclude that

    $$T(1,\pi) = \sum_i \varphi_{i\#} \bb{\theta}(1,\pi) = \frac{1}{\deg(f)} \int_{\R^n} \left( \sum_{x \in \pi^{-1}(y)} \deg(\pi,x) \right) \, d\leb^n(y)$$

    for every $\pi \in \Lip(X,\R^n)$. It follows from the additivity property of the degree that for every $y \in \R^n$ for which the preimage $\pi^{-1}(y)$ is finite we have
    $$\sum_{x \in \pi^{-1}(y)} \deg(\pi,x) = \deg(\pi).$$
    Since $X$ is closed, we have $\deg(\pi) = 0$, see e.g.~\cite[p. 269]{dold-1980}. This shows $T(1,\pi)=0$ and completes the proof.    
\end{proof}

Finally, we can give the proof of the main theorem.

\begin{proof}[Proof of Theorem~\ref{thme: main}]
    Let $T \in \bI_n(X)$ be the integral cycle given by Lemma \ref{lemma: existence-cycle-new} and $\{\varphi_i,K_i,\theta_i\}$ its parametrization. Recall that $|\theta_i(x)| =1$ for every $i \in \N$ and almost all $x \in K_i$. Therefore, by the definition of the Hausdorff or Busemann mass we have
    $$
    \mass^{\rm b}(T) =\sum_{i\in\N}\int_{K_i}|\theta_i| \,\mathbf{J}(\md\varphi_i)\,d\mathscr{L}^n =\sum_{i\in\N}\int_{K_i} \,\mathbf{J}(\md\varphi_i)\,d\mathscr{L}^n.
    $$ 
    Since $\{\varphi_i,K_i\}$ is an atlas of $X$, the area formula implies $\mass^{\rm b}(T)=\Ha^n(X)$ and in particular, $\mass^{\rm b}(T)=\Ha^n(M)$. Moreover, again using that $\{\varphi_i,K_i\}$ is an atlas of $X$, we have $\spt T = X$. By partitioning the parametrization of $T$ if necessary, and by Lemma \ref{lemma: mass-preserving-almost-injective} and Lemma \ref{lemma: f-bi-lip-decomp-new}, we may assume that the following holds: $f$ is bi-Lipschitz on each $C_i = \varphi_i(K_i)$ and the preimage $f^{-1}(y)$ consists of a single point for every $i \in \N$ and for all $y \in f(C_i)$. It follows that $\{f\circ \varphi_i,K_i,\theta_i\}$ is a parametrization of the integral cycle $f_\# T \in \bI_n(M)$. Furthermore, using in addition that $f$ is surjective, we conclude that $\{f\circ \varphi_i,K_i\}$ is an atlas of $M$. Thus, as before, the definition of the Hausdorff or Busemann mass $\mass^{\rm b}$ and the area formula imply that $\mass^{\rm b}(f_\#T)=\Ha^n(M) = \mass^{\rm b}(\bb{M})$. Since $M$ is a closed, oriented Riemannian manifold, the constancy theorem applies. That is, each integral $n$-cycle $S\in \bI_n(M)$ is of the form $S = k \cdot \bb{M}$ for some integer $k \in \Z$; see \cite[Corollary~3.13]{federer-1960}. For $S = f_\#T$, we have that $k$ is either $1$ or $-1$ because $\mass^{\rm b}(f_\#T)=\mass^{\rm b}(\bb{M})$. Therefore, by replacing $T$ by $-T$ if necessary, we may suppose that $f_\# T = \bb{M}$. We conclude that $T$ satisfies all conditions of Theorem \ref{thme: lip-vol-rigidity-züst} and hence, $f\colon X \to M$ is an isometric homeomorphism.  
\end{proof}

\bibliographystyle{plain}

\end{document}